\documentclass{amsart} \usepackage{hyperref}
\pdfoutput=1
\usepackage{enumerate} \usepackage{upref}
\usepackage{verbatim} \usepackage{color}
\usepackage{graphicx}
\newcommand*{\mailto}[1]{\href{mailto:#1}{\nolinkurl{#1}}}

\newtheorem{theorem}{Theorem}[section]
\newtheorem{lemma}[theorem]{Lemma}

\newtheorem{remark}[theorem]{Remark}

\newtheorem{proposition}[theorem]{Proposition}

\numberwithin{equation}{section}
\newtheorem{definition}[theorem]{Definition}
\allowdisplaybreaks

\newcommand{\Winf}{{W^{1,\infty}(\Real)}}
\newcommand{\epsi}{\varepsilon}
\newcommand{\Gr}{G}

\newcommand{\D}{\ensuremath{\mathcal{D}}}

\newcommand{\G}{\ensuremath{\mathcal{G}}}
\newcommand{\F}{\ensuremath{\mathcal{F}}}

\newcommand{\Real}{\mathbb{R}}

\newcommand{\E}{\mathrm{e}}

\DeclareMathOperator{\id}{Id}

\newcommand{\dott}{\,\cdot\,}

\newcommand{\sign}{\mathop{\rm sign}}
\newcommand{\nn}{\nonumber}

\newcommand{\norm}[1]{\left\Vert#1\right\Vert}
\newcommand{\norms}[1]{\Vert#1\Vert}
\newcommand{\abs}[1]{\left\vert#1\right\vert}

\def\XXint#1#2#3{{\setbox0=\hbox{$#1{#2#3}{\int}$}
    \vcenter{\hbox{$#2#3$}}\kern-.5\wd0}}

\numberwithin{equation}{section}

\begin{document}

\title[Lipschitz metric for the Camassa--Holm
equation]{Lipschitz metric for the Camassa--Holm
  equation on the line}

\author[K. Grunert]{Katrin Grunert}
\address{Faculty of Mathematics\\ University of
  Vienna\\ Nordbergstrasse 15\\ A-1090 Wien\\
  Austria}
\email{\mailto{katrin.grunert@univie.ac.at}}
\urladdr{\url{http://www.mat.univie.ac.at/~grunert/}}

\author[H. Holden]{Helge Holden}
\address{Department of Mathematical Sciences\\
  Norwegian University of Science and Technology\\
  7491 Trondheim\\ Norway\\ {\rm and} Centre of
  Mathematics for Applications\\ University of Oslo\\
  NO-0316 Oslo\\ Norway}
\email{\mailto{holden@math.ntnu.no}}
\urladdr{\url{http://www.math.ntnu.no/~holden/}}

\author[X. Raynaud]{Xavier Raynaud}
\address{Centre of Mathematics for Applications\\
  University of Oslo\\ NO-0316 Oslo\\ Norway}
\email{\mailto{xavierra@cma.uio.no}}
\urladdr{\url{http://folk.uio.no/xavierra/}}

\date{\today} \thanks{Research supported by the
  Research Council of Norway under Project
  No.~195792/V11, Wavemaker, NoPiMa, and by the Austrian Science Fund (FWF) under Grant No.~Y330.}  

\begin{abstract}
  We study stability of solutions of the Cauchy
  problem on the line for the Camassa--Holm equation
  $u_t-u_{xxt}+3uu_x-2u_xu_{xx}-uu_{xxx}=0$ with
  initial data $u_0$.  In particular, we derive a
  new Lipschitz metric $d_\D$ with the
  property that for two solutions $u$ and $v$ of
  the equation we have
  $d_\D(u(t),v(t))\le e^{Ct}
  d_\D(u_0,v_0)$. The relationship between this metric and the usual norms in $H^1$ and $L^\infty$ is clarified. 
  The method extends to the generalized hyperelastic-rod equation 
  $u_t-u_{xxt}+f(u)_x-f(u)_{xxx}+(g(u)+\frac12 f''(u)(u_x)^2)_x=0$ (for $f$ without inflection points).
\end{abstract}
\maketitle

\section{Introduction}

The Cauchy problem for the Camassa--Holm (CH) equation \cite{CH, CHH},
\begin{equation}
  \label{eq:kappaCH}
  u_t-u_{txx}+\kappa u_x+3uu_x-2u_xu_{xx}-uu_{xxx}=0,
\end{equation}
where $\kappa\in\Real$ is a constant, has
attracted much attention due to the fact that it serves as a model for shallow water waves \cite{cla} and its rich
mathematical structure. For example,  it has a bi-Hamiltonian structure, 
infinitely many conserved quantities 
and blow-up phenomena have been studied, see, e.g., \cite{cons:98}, \cite{cons:98b}, and \cite{cons:00}. 

We here focus on the construction of the Lipschitz metric for the semigroup of conservative solutions on the real line. This problem has been recently considered by Grunert, Holden, and Raynaud \cite{GHR} in the
periodic case, and here we want to present how the
approach used there has to be modified 
in the non-periodic case.

For simplicity, we will only discuss the case
$\kappa=0$, that is,
\begin{equation}\label{eq:CH}
 u_t-u_{txx}+3uu_x-2u_xu_{xx}-uu_{xxx}=0, 
\end{equation}
and from now on we refer to \eqref{eq:CH} as the
CH equation. However, the approach presented here
can also handle the generalized hyperelastic-rod
equation, see Remark \ref{rem:genhyp}. In particular, it includes the case with nonzero $\kappa$. The
generalized hyperelastic-rod equation has been
introduced in \cite{HolRay:06a}. It is given by
\begin{equation}
  \label{eq:genhyper}
  u_t-u_{xxt}+f(u)_x-f(u)_{xxx}+(g(u)+\frac12 f''(u)(u_x)^2)_x=0   
\end{equation}
where $f$ and $g$ are smooth functions.\footnote{In addition, the function $f$ is assumed to be strictly convex or concave.} With
$f(u)=\frac{u^2}2$ and $g(u)=\kappa u+u^2$, we
recover \eqref{eq:kappaCH} for any $\kappa$.  With
$f(u)=\frac{\gamma u^2}2$ and
$g(u)=\frac{3-\gamma}{2}u^2$, we obtain the
hyperelastic-rod wave equation:
\begin{equation*}
u_t-u_{txx}+3uu_x-\gamma(2u_xu_{xx}+uu_{xxx})=0,
\end{equation*}
which has been introduced by Dai
\cite{Dai_exact:98,Dai:98,Dai:2000}.

Equation \eqref{eq:CH} can be rewritten as the
following system
\begin{align}
 u_t+uu_x+P_x&=0,\\ 
P-P_{xx}=u^2+&\frac{1}{2}u_x^2,
\end{align}
where we choose $u$ to be an element of
$H^1(\Real)$. The $H^1$ norm is preserved as
\begin{equation}
  \label{eq:presH1}
 \frac{d}{dt}\norm{u(t)}_{H^1(\Real)}^2=\frac{d}{dt}\int_\Real (u^2+u_x^2)dx=0, 
\end{equation}
for any smooth solution $u$. However, even for
smooth initial data, the solution may break down
in finite time. In this case, the solution
experiences wave breaking (\cite{CHH,cons:98}):
The solution remains bounded while, at some point,
the spatial derivative $u_x$ tends to
$-\infty$. This phenomenon can be nicely
illustrated by the so called multipeakon
solutions. These are solutions of the form
\begin{equation}
\label{eq:chP}
u(t,x)=\sum_{i=1}^{n} p_i(t)e^{-\abs{x-q_i(t)}}.
\end{equation}
Let us consider the case with $n=2$ and one peakon
$p_1(0)>0$ (moving to the right) and one
antipeakon $p_2(0)<0$ (moving to the left). In the
symmetric case ($p_1(0)=-p_2(0)$ and
$q_1(0)=-q_2(0)<0$) the solution $u$ will vanish
pointwise at the collision time $t^*$ when
$q_1(t^*)=q_2(t^*)$, that is, $u(t^*,x)=0$ for all
$x\in \Real$. At time $t=t^*$, the whole energy is
concentrated at the origin, and we have $\lim_{t\to
  t^*}(u^2+u_x^2)\,dx=\norm{u_0}_{H^1}\delta$, with $\delta$ denoting the Dirac delta distribution at the origin. In
general we have two possibilities to continue the
solution beyond wave breaking, namely to set $u$
identically equal to zero for $t>t^\star$, which is called a
dissipative solution, or to let the peakons pass
through each other, which is called a conservative
solution and which is depicted in Figure
\ref{fig:peakcol}. We are interested in the latter
case, for which solutions have been studied by
Bressan and Constantin \cite{BC} and Holden and
Raynaud \cite{HR,HolRay:06b}.  Since the
$H^1$-norm is preserved, the space $H^1$ appears
as a natural space for the semigroup of
solutions. However, the previous multipeakon
example reveals the opposite. Indeed, $u(t^*,x)=0$
for all $x\in\Real$. Thus, the trivial solution $\bar u$, that is, $\bar
u(t,x)=0$ for all $t,x\in\Real$, which is also a
conservative solution, coincides with $u$ at $t=t^*$. To define a semigroup of
conservative solutions, we therefore need more information about the solution than just its pointwise values, for instance, 
the amount and location of the energy which
concentrates on sets of zero measure.  This
justifies the introduction of the set $\D$ of Eulerian coordinates, see
Definition \ref{def:D}, for which a semigroup can
be constructed \cite{HR}.

Furthermore, the $H^1$ norm is not well suited to establish a
stability result. Consider, e.g., the sequence of
multipeakons $u^\epsi$ defined as $u^\epsi (t,x)
=u(t-\epsi,x)$, see
Figure~\ref{fig:peakcol}. Then, assuming that
$\norm{u(0)}_{H^1(\Real)}=1$, we have
\begin{equation}\nn
 \lim_{\epsi\to 0}\norm{u(0)-u^\epsi(0)}_{H^1(\Real)}=0, \quad \text{and} \quad \norm{u(t^\star)-u^\epsi(t^\star)}_{H^1(\Real)}=\norm{u^\epsi(t^\star)}_{H^1(\Real)}=1, 
\end{equation}
so that the flow is clearly discontinuous with
respect to the $H^1$ norm.

\begin{figure}
  \includegraphics[width=4cm]{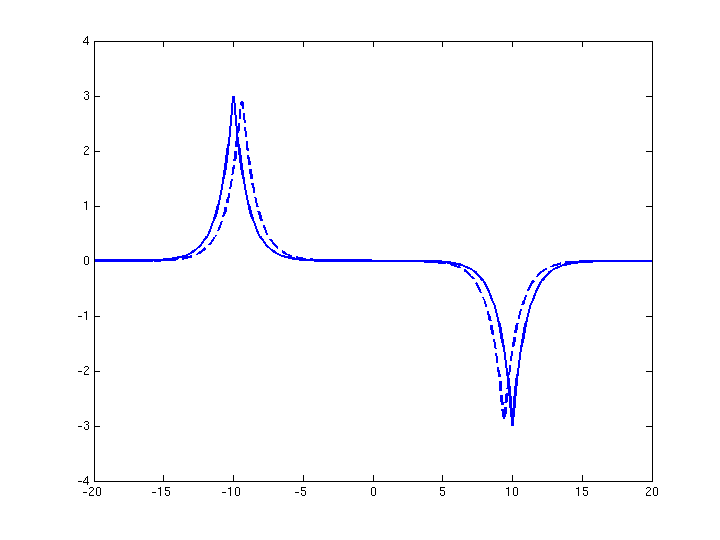}
  \includegraphics[width=4cm]{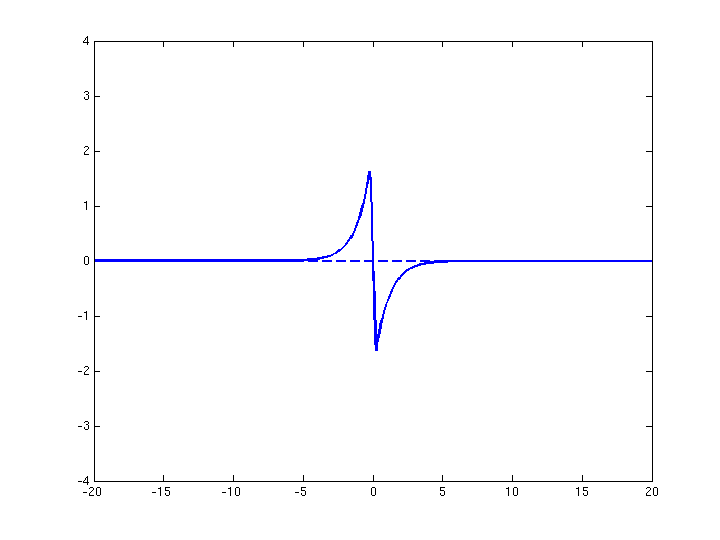}
  \includegraphics[width=4cm]{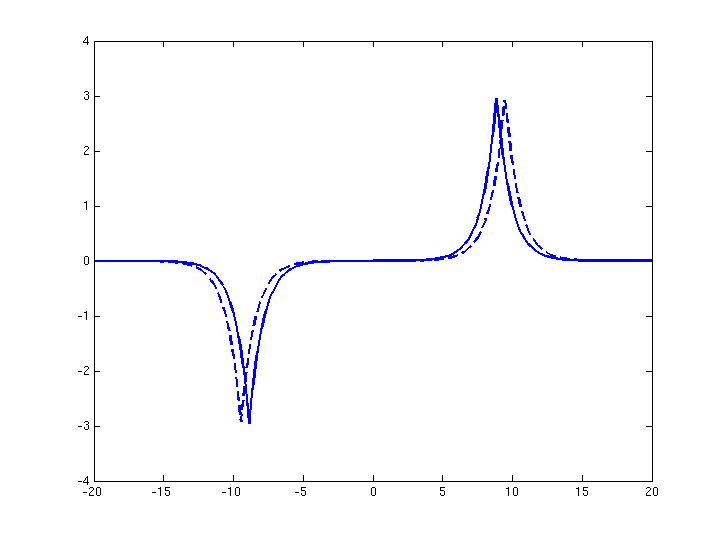}
  \label{fig:peakcol}
  \caption{The dashed curve depicts the
    antisymmetric multipeakon solution $u(t,x)$,
    which vanishes at $t^*$, for $t=0$ (left) and
    $t=t^*$ (middle) and $t=2t^*$ (right). The
    solid curve depicts the multipeakon solution
    given by $u^\epsi(t,x)=u(t-\epsi,x)$ for some
    small $\epsi>0$ (the CH equation is invariant
    with respect to time translation).}
\end{figure}

The aim of this article is to present a metric for
which the semigroup of conservative solutions on the line is
Lipschitz continuous. A more extensive discussion about Lipschitz continuity with examples from ordinary differential equations, can be found in \cite{BHR}. A detailed presentation for the Camassa--Holm equation in the periodic case is presented in \cite{GHR}, thus we here focus on explaining the differences between the periodic case and the decaying case. However, we first present the general construction.

The construction of the
metric is closely connected to the construction of
the semigroup itself. Let us outline this
construction. We rewrite the CH equation in
Lagrangian coordinates and obtain a semilinear
system of ordinary differential equations: Let
$u(t,x)$ denote the solution and $y(t,\xi)$ the
corresponding characteristics, thus
$y_t(t,\xi)=u(t,y(t,\xi))$. Then our new variables
are $y(t,\xi)$, as well as 
\begin{equation}
 U(t,\xi)=u(t,y(t,\xi)), \quad H(t,\xi)=\int_{-\infty}^{y(t,\xi)}(u^2+u_x^2)dx,  
\end{equation}
where $U$ corresponds to the Lagrangian velocity
while $H$ can be interpreted as the Lagrangian
cumulative energy distribution.  The time
evolution for any $X=(y,U,H)$ is described by
\begin{equation}
  \label{equivsystem}
  \begin{aligned}  y_t &= U,\\
    U_t &= -Q,\\ H_t &= U^3-2PU,
  \end{aligned}
\end{equation}
where 
\begin{equation}\label{repPint}
  P(t,\xi)=\frac{1}{4} \int_{\Real} \exp(-\vert y(t,\xi)-y(t,\eta)\vert ) (U^2y_\xi +H_\xi)(t,\eta)d\eta,
\end{equation}
and
\begin{equation}\label{repQint}
  Q(t,\xi)=-\frac{1}{4} \int_\Real \sign(y(t,\xi)-y(t,\eta)) \exp(-\vert y(t,\xi)-y(t,\eta)\vert ) (U^2y_\xi +H_\xi)(t,\eta)d\eta.
\end{equation}
This system is well-posed as a locally Lipschitz
system of ordinary differential equations in a Banach space, and we can define
a semigroup of solution which we denote
$S_t$. From standard  theory for ordinary differential equations we know that
$S_t$ is locally Lipschitz continuous, that is,
given $T$ and $M$,
\begin{equation}
  \label{eq:lipst}
  \norm{S_t(X_\alpha)-S_t(X_\beta)}\leq C_{M,T} \norm{X_\alpha-X_\beta}
\end{equation}
for any $t\in[0,T]$, $X_\alpha,X_\beta\in B_M=\{X\
|\ \norm{X}\leq M\}$ and where the constant
$C_{M,T}$ depends only on $M$ and $T$.

The mapping from Lagrangian to Eulerian
coordinates is surjective but not bijective. The
discrepancy between the two sets of coordinates is
due to the freedom of relabeling in Lagrangian
coordinates. The relabeling functions form a
group, which we denote $\Gr$, and which basically
consists of the diffeomorphisms of the line with
some additional assumptions (see Definition
\ref{def:Gr}). Given $X=(y,U,H)$, the element
$X\circ f=(y\circ f,U\circ f,H\circ f)$ is the
relabeled version of $X$ by the relabeling
function $f\in\Gr$. Using the fact that the
semigroup $S_t$ is equivariant with respect to
relabeling, that is,
\begin{equation}
  \label{eq:Strelinv}
  S_t(X\circ f)=S_t(X)\circ f,
\end{equation}
we can construct a semigroup of solutions on
equivalence classes from $S_t$. Finally, after
establishing the existence of a bijection between
the Eulerian coordinates and the equivalence
classes in Lagrangian coordinates, we can
transport the semigroup of solutions defined on equivalence classes and
construct a semigroup, which we denote $T_t$, of
conservative solutions in $\D$.

We want to find a metric which makes $T_t$
Lipschitz continuous. For that purpose, we
introduce a pseudometric\footnote{By a pseudometric
  we mean a map $d\colon X\times X\to[0,\infty)$ which is symmetric, $d(x,y)=d(y,x)$, for which the triangle inequality $d(x,y)\le d(x,z)+d(z,y)$ holds, and satisfies $d(x,x)=0$ for $x,y,z\in X$.} in
Lagrangian coordinates which does not distinguish
between elements of the same equivalence class and
which, at the same time, leaves the semigroup
$S_t$ locally Lipschitz continuous. This strategy
has been used in \cite{BHR} for the Hunter--Saxton
equation and in \cite{GHR} for the Camassa--Holm
equation in the periodic case. In \cite{BHR}, the
pseudometric is defined by using ideas from
Riemannian geometry. Here, we follow the approach
of \cite{GHR} and first introduce a
pseudosemimetric\footnote{By a pseudosemimetric we mean a map $d\colon X\times X\to[0,\infty)$ which is symmetric, $d(x,y)=d(y,x)$ and satisfies  $d(x,x)=0$  for $x,y\in X$.} which also identifies
elements of the same equivalence class and leaves
$S_t$ Lipschitz continuous. A natural choice,
which was applied in \cite{GHR}, is to consider the
pseudometric $\tilde J$ defined as
\begin{equation}
  \label{eq:defJt}
  \tilde J(X_\alpha,X_\beta)=\inf_{f,g\in\Gr}\norm{X_\alpha\circ f-X_\beta\circ g}.
\end{equation}
The pseudometric $\tilde J$ identifies elements of
the same equivalence class, as $\tilde J(X,X\circ
f)=0$. Moreover, it is invariant with respect to
relabeling, that is, $\tilde J(X_\alpha\circ
f,X_\beta\circ g)=\tilde J(X_\alpha,X_\beta)$ for
any $f,g\in\Gr$. It remains to prove that the
pseudosemimetric makes the semigroup $S_t$ locally
Lipschitz, that is, given $M$ and $T$, there exists
a constant $C$ depending on $M$ and $T$ such that
\begin{equation}
  \label{eq:stabjt}
  \tilde J(S_tX_\alpha,S_tX_\beta)\leq C\tilde J(X_\alpha,X_\beta),
\end{equation}
for all $t\in[0,T]$ and $X_\alpha,X_\beta\in
B_M$. The proof follows almost directly from the
stability and equivariance of $S_t$. We outline it
here. For every $\epsi>0$, there exist $f,g\in\Gr$
such that $\tilde
J(X_\alpha,X_\beta)\geq\norm{X_\alpha\circ
  f-X_\beta\circ g}-\epsi$ and we get
\begin{subequations}
  \label{eq:notusstab}
  \begin{align}
    \tilde J(S_tX_\alpha,S_tX_\beta)&\leq\norm{S_t(X_\alpha)\circ f-S_t(X_\beta)\circ g}&\\
    &=\norm{S_t(X_\alpha\circ f)-S_t(X_\beta\circ
      g)}&
    \text{ (as $S_t$ is equivariant)}\\
    \label{eq:notusstab3}
    &\leq C_M\norm{X_\alpha\circ f-X_\beta\circ g}&\text{ (by \eqref{eq:lipst})}\\
    &\leq C_M(\tilde J(X_\alpha,X_\beta)+\epsi)&
  \end{align}
\end{subequations}
and \eqref{eq:stabjt} follows by letting $\epsi$
tend to zero. However, the use of the Lipschitz
stability of $S_t$ \eqref{eq:lipst} relies on
bounds on $\norm{X_\alpha\circ f}$ and
$\norm{X_\beta\circ g}$ that are 
unavailable. The problem is that the norm
$\norm{\dott}$ of the Banach space is \textit{not}
invariant with respect to relabeling and
therefore, since $f$ and $g$ are a priori
arbitrary, we cannot obtain any bound depending on
$M$ for $\norm{X_\alpha\circ f}$ and
$\norm{X_\beta\circ g}$. This motivates the
introduction in this paper of the pseudosemimetric
$J$ defined as
\begin{equation}
  J(X_\alpha,X_\beta)=\inf_{f_1,f_2\in\Gr}\big(\norm{X_\alpha\circ f_1-X_\beta}+\norm{X_\alpha-X_\beta\circ f_2}\big).
\end{equation}
As expected, the pseudosemimetric $J$ identifies
equivalence classes (we have $J(X,X\circ f)=0$)
but we lose the nice relabeling invariance
property. At the same time, this definition of $J$
implies some implicit restrictions on the
diffeomorphisms $f_1$ and $f_2$ which allow us to
bound the relabeled versions $\norm{X_\alpha\circ
  f_1}$ and $\norm{X_\beta\circ f_2}$ so that the
approach sketched in \eqref{eq:notusstab} can be
carried out.

It remains to explain why, in the periodic case
\cite{GHR}, we could use the definition of $\tilde
J$, which is a more natural definition and
moreover simplifies the proofs. In the periodic
case (we take the period equal to one), the
stability of the semigroup $S_t$ is established in
the space $W^{1,1}([0,1])$ equipped with the norm
\begin{equation}
  \label{eq:normW11}
  \norm{U}_{W^{1,1}([0,1])}=\norm{U}_{L^\infty([0,1])}+\norm{U_\xi}_{L^1([0,1])}.
\end{equation}
Note that, in order to keep these formal
explanations as simple as possible, we just
consider the second component of
$X=(y,U,H)$. Since $\norm{U\circ
  f}_{L^\infty}=\norm{U}_{L^\infty}$ and
\begin{equation*}
  \norm{U\circ
    f}_{L^1}=\int_0^1U\circ ff_\xi\,d\xi=\norm{U}_{L^1}
\end{equation*}
we have $\norm{U\circ
  f}_{W^{1,1}}=\norm{U}_{W^{1,1}}$, for any
$f\in\Gr$, so that the norm defined in
\eqref{eq:normW11} is relabeling
invariant. Now, if
the norm of the Banach space is relabeling
invariant, we have
\begin{equation}
  \label{eq:JJtequiv}
  \tilde J(X_\alpha,X_\beta)\leq J(X_\alpha,X_\beta)\leq 2\tilde J(X_\alpha,X_\beta),
\end{equation}
and the pseudosemimetrics $J$ and $\tilde J$ are
equivalent. However, the natural Banach space for
$U$ is not $W^{1,1}([0,1])$ but
$W^{1,2}([0,1])=H^1([0,1])$. In the periodic case,
it is not an issue as $H^1([0,1])\subset
W^{1,1}([0,1])$ but the corresponding embedding
does not hold in the case of the real line. This
also shows that the approach that we present here
for the real line can also be used in the periodic
case and that the novelty in this article is that
we handle a norm which is not relabeling
invariant.

The final step consists of deriving a pseudometric
from the pseudosemimetric $J$. This can be
achieved by the following general construction: Let
\begin{equation*}
  d(X_\alpha,X_\beta)=\inf \sum_{i=1}^NJ(X_{n-1},X_n),
\end{equation*}
where the infimum is taken over all finite sequences
$\{X_n\}_{n=0}^N$ with $X_0=X_\alpha$ and
$X_N=X_\beta$.  The pseudometric $d$ inherits the
Lipschitz stability property \eqref{eq:stabjt} from 
$J$. Finally, identifying elements belonging to the same equivalence class, 
the pseudometric $d$ turns into a metric on the set
of equivalence classes. By bijection, it yields a
metric in $\D$ which makes the semigroup of
conservative solutions Lipschitz continuous.

In the last section, Section \ref{sec:top}, we compare this new
metric with the usual norms in $H^1$ and
$L^\infty$.

\section{Semigroup of solutions in Lagrangian
  coordinates}
\label{sec:semilag}

In this section, we recall from \cite{HR} the
construction of the semigroup in Lagrangian
coordinates. The Camassa--Holm equation reads
\begin{equation}
  u_t-u_{xxt}+3uu_x-2u_xu_{xx}-uu_{xxx}=0,
\end{equation}
and can be rewritten as the following system
\begin{equation}
  u_t+uu_x+P_x=0,
\end{equation}
\begin{equation}\label{eq:P}
  P-P_{xx}=u^2+\frac{1}{2}u_x^2.
\end{equation}

Next, we rewrite the equation in Lagrangian
coordinates. Therefore we introduce the
characteristics
\begin{equation}
  y_t(t,\xi)=u(t,y(t,\xi)).
\end{equation}
The Lagrangian velocity $U$ reads
\begin{equation}
  U(t,\xi)=u(t,y(t,\xi)).
\end{equation}
We define the Lagrangian cumulative energy as
\begin{equation}
  H(t,\xi)=\int_{-\infty}^{y(t,\xi)} (u^2+u_x^2) dx.
\end{equation}
As an immediate consequence of the definition of
the characteristics we obtain
\begin{equation}
  U_t(t,\xi)=u_t(t,y)+y_t(t,\xi)u_x(t,y)=-P_x\circ y(t,\xi).
\end{equation}
The last term can be expressed uniquely in terms
of $y$, $U$, and $H$. From \eqref{eq:P} we obtain
the following explicit expression for $P$,
\begin{equation}
  P(t,x)=\frac{1}{2} \int_{\Real} e^{-\vert x-z \vert} \Big(u^2(t,z)+\frac{1}{2} u_x^2 (t,z)\Big)dz.
\end{equation}
Setting $Q(t,\xi)=P_x(t,y(t,\xi))$ and writing
$P(t,\xi)=P(t,y(t,\xi))$, we obtain
\begin{equation}\label{repP}
  P(t,\xi)=\frac{1}{4} \int_{\Real} \exp(-\vert y(t,\xi)-y(t,\eta)\vert ) (U^2y_\xi +H_\xi)(t,\eta)d\eta,
\end{equation}
and
\begin{equation}\label{repQ}
  Q(t,\xi)=-\frac{1}{4} \int_\Real \sign(y(t,\xi)-y(t,\eta)) \exp(-\vert y(t,\xi)-y(t,\eta)\vert ) (U^2y_\xi +H_\xi)(t,\eta)d\eta.
\end{equation}
Moreover we introduce another variable $\zeta
(t,\xi)=y(t,\xi)-\xi$.  Thus we have derived a new
system of equations, which is up to that point
only formally equivalent to the Camassa--Holm
equation:
\begin{equation}
  \label{equivsys}
  \begin{aligned}  \zeta_t &= U,\\
    U_t &= -Q,\\ H_t &= U^3-2PU.
  \end{aligned}
\end{equation}
Let $V$ be the Banach space defined by
\begin{equation*}
V=\{f\in C_b(\Real) \ |\ f_\xi\in L^2\}
\end{equation*}
where $C_b(\Real)=C(\Real)\cap L^\infty$ and the
norm of $V$ is given by
$\norm{f}_V=\norm{f}_{L^\infty}+\norm{f_\xi}_{L^2}$. Of
course $H^1\subset V$ but the converse is not
true as $V$ contains functions that do not vanish
at infinity. We will employ the Banach space $E$
defined by
\begin{equation*}
E=V\times H^1\times V
\end{equation*}
with the following norm $\norm{X}=\norm{\zeta}_{V}+\norm{U}_{H^1(\Real)}+\norm{H}_{V}$ for any $X=(\zeta, U, H)\in E$.

\begin{definition}
\label{def:F}
The set $\G$ is composed of all $(\zeta,U,H)\in E$
such that
\begin{subequations}
\label{eq:lagcoord}
\begin{align}
\label{eq:lagcoord1}
&(\zeta,U,H)\in \left[W^{1,\infty}\right]^3,\\
\label{eq:lagcoord2}
&y_\xi\geq0, H_\xi\geq0, y_\xi+H_\xi>0 
\text{  almost everywhere, and  }\lim_{\xi\rightarrow-\infty}H(\xi)=0,\\
\label{eq:lagcoord3}
&y_\xi H_\xi=y_\xi^2U^2+U_\xi^2\text{ almost everywhere},
\end{align}
\end{subequations}
where we denote $y(\xi)=\zeta(\xi)+\xi$.
\end{definition}
Given a constant $M>0$, we denote by $B_M$ the
ball
\begin{equation}
  \label{eq:defBM}
  B_M=\{X\in E\ |\ \norm{X}\leq M\}.
\end{equation}
\begin{theorem}
\label{th:global}
For any $\bar X=(\bar y,\bar U,\bar H)\in\G$, the
system \eqref{equivsys} admits a unique global
solution $X(t)=(y(t),U(t),H(t))$ in
$C^1(\Real_+,E)$ with initial data $\bar X=(\bar
y,\bar U,\bar H)$. We have $X(t)\in\G$ for all
times. If we equip $\G$ with the topology induced
by the $E$-norm, then the mapping
$S\colon\G\times\Real_+\to\G$ defined by
\begin{equation*}
S_t(\bar X)=X(t)
\end{equation*}
is a continuous semigroup. More precisely, given
$M>0$ and $T>0$, there exists a constant $C_M$
which depends only on $M$ and $T$ such that, for
any two elements $X_\alpha,X_\beta\in\G\cap B_M$,
we have
\begin{equation}
  \label{eq:stabSt}
  \norm{S_tX_\alpha-S_tX_\beta}\leq C_M\norm{X_\alpha-X_\beta}
\end{equation}
for any $t\in[0,T]$.
\end{theorem}

\begin{proof}
  We have
  \begin{equation*}
    X_t=F(X)
  \end{equation*}
  where $F$ is locally Lipschitz (see proof
  \cite[Theorem 2.3]{HR}). It implies, by
  Gronwall's lemma, that we have
  \begin{equation*}
    \norm{S_t(X_\alpha)-S_t(X_\beta)}\leq C\norm{X_\alpha-X_\beta}
  \end{equation*}
  for $t\in[0,T]$, where the constant $C$ only
  depends on
  $\sup_{t\in[0,T]}\norm{S_t(X_\alpha)}$ and
  $\sup_{t\in[0,T]}\norm{S_t(X_\beta)}$. In
  \cite[Theorem 2.8]{HR} it is proved that
  $\sup_{t\in[0,T]}\norm{S_t(X_\alpha)}$ only
  depends of $\norm{X_\alpha}$ and $T$,  and thus
  on $M$ and $T$.
\end{proof}

\begin{definition}
  \label{def:Gr}
We denote by $\Gr$ the
subgroup of the group of homeomorphisms from
$\Real$ to $\Real$ such that
\begin{subequations}
\label{eq:Hcond}
\begin{align}
 \label{eq:Hcond1}
f-\id\text{ and }f^{-1}-\id &\text{ both belong to }W^{1,\infty}(\Real),\\ 
\label{eq:Hcond2}
f_\xi-1& \text{ belongs to }L^2(\Real),
\end{align}
\end{subequations}
where $\id$ denotes the identity function. Given
$\kappa>0$, we denote by $\Gr_\kappa$ the
subset $\Gr_\kappa$ of $\Gr$ defined by
\begin{equation*}
  \Gr_\kappa=\{f\in\Gr\ |\ \norm{f-\id}_{\Winf}+\norm{f^{-1}-\id}_{\Winf}\leq\kappa\}.
\end{equation*} 
\end{definition}
The subsets $\Gr_\kappa$ do not possess  the group structure
of $G$.  The next lemma provides a useful
characterization of $\Gr_\kappa$.
\begin{lemma}[{\cite[Lemma 3.2]{HR}}] 
\label{lem:charH}
Let $\kappa\geq0$. If $f$ belongs to $\Gr_\kappa$,
then $1/(1+\kappa)\leq f_\xi\leq 1+\kappa$ almost
everywhere. Conversely, if $f$ is absolutely
continuous, $f-\id\in\Winf$, $f$ satisfies
\eqref{eq:Hcond2} and there exists $c\geq 1$ such
that $1/c\leq f_\xi\leq c$ almost everywhere, then
$f\in\Gr_\kappa$ for some $\kappa$ depending only
on $c$ and $\norm{f-\id}_{\Winf}$.
\end{lemma}
We define the subsets $\F_\kappa$ and $\F$ of $\G$
as follows
\begin{equation*}
\F_\kappa=\{X=(y,U,H)\in\G\ |\ y+H\in \Gr_\kappa\},
\end{equation*}
and
\begin{equation*}
\F=\{X=(y,U,H)\in\G\ |\ y+H\in \Gr\}.
\end{equation*}
For $\kappa=0$, $\Gr_0=\{\id\}$. As we shall see,
the space $\F_0$ will play a special role. These
sets are relevant only because they are preserved
by the governing equation \eqref{equivsys} as the
next lemma shows. In particular, while the mapping
$\xi\mapsto y(t,\xi)$ may not be a diffeomorphism
for some time $t$, the mapping $\xi\mapsto
y(t,\xi)+H(t,\xi)$ remains a diffeomorphism for
all times $t$.
\begin{lemma} 
\label{lem:Fpres}
The space $\F$ is preserved by the governing
equation \eqref{equivsys}. More precisely, given
$\kappa,T\geq0$, there exists $\kappa'$ which only
depends on $T$, $\kappa$ and $\norm{\bar X}$
such that
\begin{equation*}
S_t(\bar X)\in\F_{\kappa'}
\end{equation*}
for any $\bar X\in\F_{\kappa}$.
\end{lemma}
\begin{proof}
  Let $\bar X=(\bar y, \bar U , \bar
  H)\in\F_\kappa$, we denote by
  $X(t)=(y(t),U(t),H(t))$ the solution of
  \eqref{equivsys} with initial data $\bar X$ and
  set $h(t,\xi)=y(t,\xi)+H(t,\xi)$, $\bar
  h(\xi)=\bar y(\xi)+\bar H(\xi)$. By definition,
  we have $\bar h\in\Gr_\kappa$ and, from
  Lemma~\ref{lem:charH}, $1/c\leq \bar h_\xi\leq
  c$ almost everywhere, for some constant $c>1$
  depending only on $\kappa$. We consider a fixed
  $\xi$ and drop it in the notation. Applying
  Gronwall's inequality backward in time to
  \eqref{equivsys} we obtain
  \begin{equation}\label{eq:Fpres}
    \vert y_\xi(0)\vert +\vert H_\xi(0)\vert +\vert U_\xi(0)\vert \leq \E^{CT}(\vert y_\xi(t)\vert +\vert H_\xi(t)\vert +\vert U_\xi(t)\vert ), 
  \end{equation}
  for some constant $C$ which depends on
  $\norm{X(t)}_{C([0,T],E)}$, which itself depends
  only on $\norm{\bar X}$ and $T$. From
  \eqref{eq:lagcoord3}, we have
  \begin{equation}
    \vert U_\xi(t)\vert \leq \sqrt{y_\xi(t)H_\xi(t)}\leq \frac{1}{2}(y_\xi(t)+H_\xi(t)).
  \end{equation}
  Hence, since $y_\xi$ and $H_\xi$ are positive, \eqref{eq:Fpres} gives us 
  \begin{equation}
    \frac{1}{c}\leq \bar y_\xi+\bar H_\xi\leq \frac{3}{2}\E^{CT}(y_\xi(t)+H_\xi(t)), 
  \end{equation}
  and $h_\xi(t)=y_\xi(t)+H_\xi(t)\geq
  \frac{2}{3c}\E^{-CT}$. Similarly, by applying
  Gronwall's lemma forward in time, we obtain
  $y_\xi(t)+H_\xi(t)\leq\frac{3}{2}c\, \E^{CT}$. We
  have $\norm{(y+H)(t)-\xi}_{L^\infty(\Real)}\leq
  \norm{X(t)}_{C([0,T],E)}\leq C$ and
  $\norm{y_\xi+H_\xi-1}_{L^2}\leq\norm{\zeta_\xi}_{L^2}+\norm{H_\xi}_{L^2}\leq
  C$ for another constant $C$ which only depends
  on $\norm{\bar X}$ and $T$. Hence, applying
  Lemma~\ref{lem:charH}, we obtain that
  $y(t,\dott)+H(t,\dott)\in G_{\kappa^\prime}$ and
  therefore $X(t)\in F_{\kappa^\prime}$ for some
  $\kappa^\prime$ depending only on $\kappa$, $T$,
  and $\norm{\bar X}$.
\end{proof}

For the sake of simplicity, for any
$X=(y,U,H)\in\F$ and any function $f\in\Gr$, we
denote $(y\circ f,U\circ f,H\circ f)$ by $X\circ
f$. This operation corresponds to relabeling.

\begin{definition}
  We denote by $\Pi(X)$ the projection of $\F$
  into $\F_0$ defined as
  \begin{equation*}
    \Pi(X)=X\circ(y+H)^{-1}
  \end{equation*}
  for any $X=(y,U,H)\in\F$.
\end{definition}
The element $\Pi(X)$ is the unique relabeled
version of $X$ that belongs to $\F_0$.
\begin{lemma} 
  \label{lem:equivPi}
  The mapping $S_t$ is equivariant, that is, 
  \begin{equation*}
    S_t(X\circ f)=S_t(X)\circ f.
  \end{equation*}
\end{lemma}
This follows from the governing equation and the
equivariance of the mappings $X\mapsto P(X)$ and
$X\mapsto Q(X)$, where $P$ and $Q$ are defined in
\eqref{repP} and \eqref{repQ}, see \cite{HR} for
more details. From this lemma we get that
\begin{equation}
  \label{eq:PiSt}
  \Pi\circ S_t\circ \Pi=\Pi\circ S_t.
\end{equation}

\begin{definition}
  We define the semigroup $\bar S_t$ on $\F_0$ as
  \begin{equation*}
    \bar S_t=\Pi\circ S_t.
  \end{equation*}
\end{definition}

The semigroup property of $\bar S_t$ follows from
\eqref{eq:PiSt}. From \cite{HR}, we know that
$\bar S_t$ is continuous with respect to the norm
of $E$. It follows basically from the continuity of
the mapping $\Pi$,  but $\Pi$ is not Lipschitz
continuous and the goal of the next section is to
find a metric that makes $\bar S_t$ Lipschitz
continuous.

\begin{remark} 
  \label{rem:genhyp}
  The details of the construction of the semigroup
  of solutions in Lagrangian coordinates for the
  generalized hyperelastic-rod equation
  \eqref{eq:genhyper} is given in
  \cite{HolRay:06a}. The construction is based on
  a reformulation of the equation in Lagrangian
  coordinates which leads to a semilinear system
  of equations, similar to \eqref{equivsys} for
  the Camassa--Holm equation. In the case of the
  generalized hyperelastic-rod equation, the
  equation can be rewritten as
  \begin{equation}
    \left\{
      \begin{aligned}
        \label{eq:sys}
        \zeta_t&=f'(U),\\
        U_t&=-Q,\\
        H_t&=G(U)-2PU,
      \end{aligned}
    \right.
  \end{equation}
  where $G(v)$ is given by
  \begin{equation}
    \label{eq:Gdef}
    G(v)=\int_0^v(2g(z)+f''(z)z^2)\,dz,
  \end{equation}
  and
  \begin{multline}
    \label{eq:Qrod}
    Q(t,\xi)=-\frac{1}{2}\int_\Real\sign(\xi-\eta)\exp\big(-\sign(\xi-\eta)(y(\xi)-y(\eta))\big)\\
    \times\Big(\big(g(U)-\frac12f''(U)U^2\big)y_\xi+\frac12f''(U)H_\xi\Big)(\eta)\,d\eta,
  \end{multline}
  \begin{multline}
    \label{eq:Prod}
    P(t,\xi)=\frac{1}{2}\int_\Real\exp\big(-\sign(\xi-\eta)(y(\xi)-y(\eta))\big)\\
    \times\Big(\big(g(U)-\frac12f''(U)U^2\big)y_\xi+\frac12f''(U)H_\xi\Big)(\eta)\,d\eta.
  \end{multline}
  Section \ref{sec:semilag} outlines the
  construction of the semigroup of solutions in
  Lagrangian coordinates. The construction of the
  metric in Eulerian coordinates, which is given
  in the following sections, relies basically on
  two fundamental results of this section: The
  Lipschitz stability of the semigroup of solution
  in Lagrangian coordinates (Theorem
  \ref{th:global}) and the equivariance of the
  semigroup (Lemma \ref{lem:equivPi}). The same
  results hold for the generalized hyperelastic-rod equation, see \cite[Theorem 2.8 and Theorem
  3.6]{HolRay:06a} so that it is possible to
  define a Lipschitz stable metric for this
  equation in the same way as we do it for the CH
  equation.

\end{remark}

\section{Lipschitz metric for the semigroup $\bar S_t$}

\begin{definition}
  \label{def:J}
  Let $X_\alpha,X_\beta\in\F$, we define
  $J(X_\alpha,X_\beta)$ as
  \begin{equation}
    \label{eq:defJ}
    J(X_\alpha,X_\beta)=\inf_{f_1,f_2\in\Gr}\big(\norm{X_\alpha\circ f_1-X_\beta}+\norm{X_\alpha-X_\beta\circ f_2}\big).
  \end{equation}
\end{definition}
The mapping $J$ is symmetric. Moreover, if
$X_\alpha$ and $X_\beta$ are equivalent, then
$J(X_\alpha,X_\beta)=0$. Our goal is to create a
distance between equivalence classes, and that is
the reason why we
introduce the pseudosemimetric $\tilde J$ as follows in the periodic case (\cite{GHR}).
\begin{definition}
  \label{def:Jt}
  Let $X_\alpha,X_\beta\in\F$, we define
  $\tilde J(X_\alpha,X_\beta)$ as
  \begin{equation*}
    \tilde J(X_\alpha,X_\beta)=\inf_{f,g\in\Gr}\norm{X_\alpha\circ f-X_\beta\circ g}.
  \end{equation*}
\end{definition}
The pseudosemimetric $\tilde J$ is relabeling
invariant, that is, $\tilde J(X_\alpha\circ
f,X_\beta\circ g)=\tilde
J(X_\alpha,X_\beta)$. With Definition \ref{def:J},
we lose this important property. However,
Definition \ref{eq:defJ} allows us to obtain
estimates that cannot be obtained by Definition
\ref{def:Jt}, see the proof of Theorem
\ref{th:stab}. In addition, it turns out that we
do not actually need the relabeling invariance
property to hold strictly and the estimates
contained in the following lemma are enough for
our purpose.
\begin{lemma}
  \label{lem:Jrelabbound} Given
  $X_\alpha,X_\beta\in \F$ and $f\in\Gr_\kappa$,
  we have
  \begin{equation}
    \label{eq:xrlabnoem}
    \norm{X_\alpha\circ f-X_\beta\circ f}\leq C\norm{X_\alpha-X_\beta}
  \end{equation}
  so that 
  \begin{equation}
    \label{eq:Jrelabbound}
    J(X_\alpha\circ f,X_\beta)\leq CJ(X_\alpha,X_\beta)
  \end{equation}
  for some constant $C$ which depends only on
  $\kappa$.
\end{lemma}
\begin{proof}
  Let us prove \eqref{eq:xrlabnoem}. Let $\bar
  X_\alpha=X_\alpha\circ f$ and $\bar
  X_\beta=X_\beta\circ f$. We have $\bar\zeta_\alpha=\bar y_\alpha-\id$ and
  $\bar\zeta_\beta=\bar y_\beta-\id$ so that
  $\norm{\bar\zeta_\alpha-\bar\zeta_\beta}_{L^\infty}=\norm{\bar
    y_\alpha-\bar y_\beta}_{L^\infty}=\norm{
    y_\alpha- y_\beta}_{L^\infty}=\norm{
    \zeta_\alpha- \zeta_\beta}_{L^\infty}$. Hence,
  $\norm{\bar X_\alpha-\bar
    X_\beta}_{L^\infty}=\norm{X_\alpha-X_\beta}_{L^\infty}$.
By definition we have 
$\bar y_\alpha(\xi)= y_\alpha(f(\xi))= f(\xi)+\zeta_\alpha(f(\xi))=\xi+\bar \zeta_\alpha(\xi)$ and hence $\bar\zeta_\alpha(\xi)=\zeta_\alpha(f(\xi))+f(\xi)-\xi$. Thus
  \begin{align*}
    \norm{\bar\zeta_{\alpha,\xi}-\bar\zeta_{\beta,\xi}}_{L^2}^2&=\norm{\zeta_{\alpha,\xi}\circ ff_\xi-\zeta_{\beta,\xi}\circ ff_\xi}_{L^2}^2\\
    &=\int_\Real(\zeta_{\alpha,\xi}-\zeta_{\beta,\xi})^2(f(\xi))f_\xi^2(\xi)\,d\xi\\
    &\leq (1+\kappa)\int_\Real(\zeta_{\alpha,\xi}-\zeta_{\beta,\xi})^2(f(\xi))f_\xi(\xi)\,d\xi\\
    &\leq (1+\kappa)\int_\Real(\zeta_{\alpha,\xi}-\zeta_{\beta,\xi})^2(\xi)\,d\xi\\
    &\leq (1+\kappa)\norm{\zeta_{\alpha,\xi}-\zeta_{\beta,\xi}}_{L^2}
  \end{align*}
  so that $\norm{\bar X_{\alpha,\xi}-\bar
    X_{\beta,\xi}}_{L^2}\leq
  C\norm{X_{\alpha,\xi}-X_{\beta,\xi}}_{L^2}$. We
  have
  \begin{align}
    \notag
    \norm{\bar U_\alpha-\bar U_\beta}_{L^2}^2&=\int_{\Real}(U_\alpha-U_\beta)^2\circ f(\xi)\,d\xi\\
    \label{eq:estUcircfL2}
    &\leq(1+\kappa)\int_{\Real}(U_\alpha-U_\beta)^2\circ
    f
    f_\xi\,d\xi=(1+\kappa)\norm{U_\alpha-U_\beta}_{L^2}^2.
  \end{align}
  This concludes the proof of
  \eqref{eq:xrlabnoem}. For any $f\in\Gr_\kappa$
  and any $f_1,f_2\in\Gr$, we have
  \begin{align*}
    J(X_\alpha\circ f,X_\beta)&\leq
    \norm{X_\alpha\circ f\circ f_1-X_\beta}+\norm{X_\alpha\circ f-X_\beta\circ f_2}\\
    &\leq \norm{X_\alpha\circ f\circ
      f_1-X_\beta}+C\norm{X_\alpha-X_\beta\circ
      f_2\circ f^{-1}}.
  \end{align*}
  Hence, after taking $C\geq 1$,
  \begin{equation*}
    J(X_\alpha\circ f,X_\beta)\leq C(\norm{X_\alpha\circ f\circ f_1-X_\beta}+\norm{X_\alpha-X_\beta\circ f_2\circ f^{-1}}),
  \end{equation*}
  which implies, after taking the infimum,
  \begin{equation*}
    J(X_\alpha\circ f,X_\beta)\leq C\inf_{f_1,f_2\in\Gr}\big(\norm{X_\alpha\circ f_1-X_\beta}+\norm{X_\alpha-X_\beta\circ f_2}\big).
  \end{equation*}
\end{proof}
From the pseudosemimetric $J$, we obtain a
metric $d$ by the following construction.
\begin{definition}
  Let $X_\alpha,X_\beta\in\F_0$, we define
  $d(X_\alpha,X_\beta)$ as
  \begin{equation}
    \label{eq:defdist}
    d(X_\alpha,X_\beta)=\inf \sum_{i=1}^NJ(X_{n-1},X_n)
  \end{equation}
  where the infimum is taken over all sequences
  $\{X_n\}_{n=0}^N\in\F_0$ which satisfy
  $X_0=X_\alpha$ and $X_N=X_\beta$.
\end{definition}
\begin{lemma}
  \label{lem:LinfbdJ}
  For any $X_\alpha,X_\beta\in\F_0$, we have
  \begin{equation}
    \label{eq:LinfbdJ}
    \norm{X_\alpha-X_\beta}_{L^\infty}\leq 2 d(X_\alpha,X_\beta).
  \end{equation}
\end{lemma}
\begin{proof}
  First, we prove that, for any
  $X_\alpha,X_\beta\in\F_0$, we have
  \begin{equation}
    \label{eq:linfcompj}
    \norm{X_\alpha-X_\beta}_{L^\infty}\leq 2 J(X_\alpha,X_\beta).
  \end{equation}
  We have
  \begin{align}
    \notag
    \norm{X_\alpha-X_\beta}_{L^\infty}&\leq\norm{X_\alpha-X_\alpha\circ
      f}_{L^\infty}+\norm{X_\alpha\circ f-X_\beta}_{L^{\infty}}\\
    \label{eq:xamxbediffg}
    &\leq
    \norm{X_{\alpha,\xi}}_{L^\infty}\norm{f-\id}_{L^\infty}+\norm{X_\alpha\circ
      f-X_\beta}_{L^{\infty}}.
  \end{align}
  It follows from the definition of $\F_0$ that
  $0\leq y_\xi\leq1$, $0\leq H_\xi\leq1$ and
  $\abs{U_\xi}\leq1$ so that
  $\norm{X_{\alpha,\xi}}_{L^\infty}\leq 3$. We
  also have
  \begin{equation*}
    \norm{f-\id}_{L^\infty}=\norm{(y_\alpha+H_\alpha)\circ f-(y_\beta+H_\beta)}_{L^\infty}\leq\norm{X_\alpha\circ f-X_\beta}_{L^\infty}.
  \end{equation*}
  Hence, from \eqref{eq:xamxbediffg}, we get
  \begin{equation*}
    \norm{X_\alpha-X_\beta}_{L^\infty}\leq4\norm{X_\alpha\circ f-X_\beta}_{L^\infty}.
  \end{equation*}
  In the same way, we obtain
  $\norm{X_\alpha-X_\beta}_{L^\infty}\leq4\norm{X_\alpha-X_\beta\circ
    f}_{L^\infty}$ for any $f\in\Gr$. After adding
  these two last inequalities and taking the
  infimum, we get \eqref{eq:linfcompj}. For any
  $\epsi>0$, we consider a sequence
  $\{X_n\}_{n=0}^N\in\F_0$ such that
  $X_0=X_\alpha$ and $X_N=X_\beta$ and
  $\sum_{i=1}^NJ(X_{n-1},X_n)\leq
  d(X_\alpha,X_\beta)+\epsi$. We have
  \begin{align*}
    \norm{X_\alpha-X_\beta}_{L^\infty}&\leq
    \sum_{n=1}^{N}\norm{X_{n-1}-X_n}_{L^\infty}\\
    &\leq2\sum_{n=1}^{N}J(X_{n-1},X_n)\\
    &\leq2(d(X_\alpha,X_\beta)+\epsi).
  \end{align*}
  After letting $\epsi$ tend to zero, we get
  \eqref{eq:LinfbdJ}.
\end{proof}

\begin{lemma}
  The mapping $d:\F_0\times\F_0\to\Real_+$ is a
  distance on $\F_0$, which is bounded as follows
  \begin{equation}
    \label{eq:dequiv}
    \frac12\norm{X_\alpha-X_\beta}_{L^\infty}\leq d(X_\alpha,X_\beta)\leq2\norm{X_\alpha-X_\beta}.
  \end{equation}
\end{lemma}
\begin{proof}
  The symmetry is embedded in the definition of
  $J$ while the construction of $d$ from $J$ takes
  care of the triangle inequality. From Lemma
  \ref{lem:LinfbdJ}, we get that
  $d(X_\alpha,X_\beta)=0$ implies
  $X_\alpha=X_\beta$. The first inequality in
  \eqref{eq:dequiv} follows from Lemma
  \ref{lem:LinfbdJ} while the second one follows
  from the definition of $J$ and $d$. Indeed, we
  have
  \begin{equation*}
    d(X_\alpha,X_\beta)\leq J(X_\alpha,X_\beta)\leq2\norm{X_\alpha-X_\beta}.
  \end{equation*}
\end{proof}

We need to introduce the subsets of bounded energy
in $\F_0$. Note that the total energy is equal to
$H(\infty)-H(-\infty)=\norm{H}_{L^\infty}$ as
$H(-\infty)=0$ and $H$ is increasing, see
Definition \ref{def:F}.

\begin{definition}
  We denote by $\F^M$ the set
  \begin{equation*}
    \F^M=\{X=(y,U,H)\in \F\ |\ \norm{H}_{L^\infty}\leq M\}
  \end{equation*}
  and
  \begin{equation*}
    \F_0^M=\F_0\cap\F^M.
  \end{equation*}
\end{definition}
The ball $B_M$ (see \eqref{eq:defBM}) is not
preserved by the equation while the set $\F^M$ is
preserved because of the conservation of 
energy, namely,
\begin{equation*}
  \norm{H(t,\cdot)}_{L^\infty}=\lim_{\xi\to\infty}H(t,\xi)=\lim_{\xi\to\infty}H(0,\xi)=\norm{H(0,\cdot)}_{L^\infty}.
\end{equation*}
The set $\F^M$ is also conserved by relabeling as,
for any $f\in\Gr$, $\norm{H\circ
  f}_{L^\infty}=\norm{H}_{L^\infty}$.  The ball
$B_M$ is included in $\F^M$ but the
reverse inclusion does not hold. However, as the
next lemma shows, when we restrict ourselves to
$\F_0$, the sets $\F_0\cap\F^M$ and $\F_0\cap B_M$
are in fact equivalent.
\begin{lemma}
  \label{lem:bdEf0M}
  For any element $X\in\F_0^M$, we have
  \begin{equation}
    \label{eq:bdEF0m}
    \F_0\cap B_M\subset \F_0^M\subset B_{\bar M}
  \end{equation}
  for some constant $\bar M$ depending only on
  $M$.
\end{lemma}
\begin{proof}
  Since $y_\xi+H_\xi=1$, $H_\xi\geq0$,
  $y_\xi\geq0$, we get $0\leq H_\xi\leq 1$ and
  $0\leq y_\xi\leq 1$. Hence,
  $\norm{H_\xi}_{L^2}^2\leq\int_\Real
  H_\xi\,d\xi=H(\infty)\leq M$. Since $\zeta=-H$,
  we get $\norm{\zeta_\xi}_{L^2}\leq M$. By
  \eqref{eq:lagcoord3}, we get $U_\xi^2\leq y_\xi
  H_\xi$ and therefore
  $\int_{\Real}U_\xi^2\,d\xi\leq H(\infty)\leq
  M$. Finally, we have to show that $\norm{U}_{L^2} \leq C(M)$. Therefore observe that by \eqref{eq:lagcoord3}, 
  $\int_\Real U^2y_\xi d\xi\leq\int_\Real H_\xi d\xi\leq M$. This together with the fact that $X\in\F_0$ yields 
  \begin{equation}\nn 
   \int_\Real U^2d\xi =\int_\Real U^2y_\xi d\xi+\int_\Real U^2H_\xi d\xi\leq M(1+\norm{U}_{L^\infty}^2). 
  \end{equation}
  Thus it is left to estimate $\norm{U}_{L^\infty}$, which can be done as follows,
  \begin{equation}\nn
   U^2(\xi)=2\int_{-\infty}^\xi U(\eta)U_\xi(\eta)d\eta=2\int_{\{\eta\leq \xi \vert y_\xi(\eta)>0\}} U(\eta)U_\xi(\eta) d\eta, 
  \end{equation}
  where we used that $U_\xi(\xi)=0$, when $y_\xi(\xi)=0$ by \eqref{eq:lagcoord3}. For almost every $\xi$ such that $y_\xi(\xi)>0$, we have 
  \begin{equation}\nn
   \vert U(\xi)U_\xi(\xi)\vert =\vert \sqrt{y_\xi(\xi)}U(\xi)\frac{U_\xi(\xi)}{\sqrt{y_\xi(\xi)}}\vert \leq \frac{1}{2}\Big(U^2(\xi)y_\xi(\xi)+\frac{U_\xi^2(\xi)}{y_\xi(\xi)}\Big)\leq \frac{1}{2} H_\xi(\xi), 
  \end{equation}
  from \eqref{eq:lagcoord3} and hence $\norm{U}_{L^\infty}^2\leq M$ and $\norm{U}_{L^2}^2 \leq M(1+M)$. 
\end{proof}

\begin{definition}\label{def:metric}
  Let $d^M$ be the distance on $\F_0^M$ which is
  defined, for any $X_\alpha,X_\beta\in\F_0^M$, as
  \begin{equation*}
    d^M(X_\alpha,X_\beta)=\inf \sum_{n=1}^NJ(X_{n-1},X_n)
  \end{equation*}
  where the infimum is taken over all the
  sequences $\{X_n\}_{n=0}^N\in\F_0^M$ which
  satisfy $X_0=X_\alpha$ and $X_N=X_\beta$.
\end{definition}

We can now prove our main stability theorem. 
\begin{theorem}
  \label{th:stab} Given $T>0$ and $M>0$, there
  exists a constant $C_M$ which depends only on $M$
  and $T$ such that, for any
  $X_\alpha,X_\beta\in\F_0^M$ and $t\in[0,T]$, we
  have
  \begin{equation}
    \label{eq:stab}
    d^M(\bar S_tX_\alpha,\bar S_tX_\beta)\leq C_Md^M(X_\alpha,X_\beta).
  \end{equation}
\end{theorem}

\begin{proof}
  By the definition of $d^M$, for any $\epsi$ such
  that $0<\epsi\leq 1 $ there exists a sequence
  $\{X_n\}_{n=0}^N$ in $\F_0^M$ such that
  $X_0=X_\alpha$, $X_N=X_\beta$,
  \begin{equation*}
    \sum_{n=1}^N J(X_{n-1},X_n)\leq d^M(X_\alpha,X_\beta)+\epsi.
  \end{equation*}
  Hence, there exist functions
  $\{f_n\}_{n=0}^{N-1}$,
  $\{\tilde{f}_n\}_{n=1}^{N}$ in $\Gr$ such that
  \begin{equation}
    \label{eq:sumXnm1}
    \sum_{n=1}^N(\norm{X_{n-1}\circ f_{n-1}-X_{n}}+\norms{X_{n-1}-X_n\circ \tilde{f}_{n}})\leq
    d^M(X_\alpha,X_\beta)+2\epsi.
  \end{equation}
  Let us denote 
  \begin{equation*}
    X_n^t=S_t(X_n),\quad g_n^t=y_n^t+H_n^t,\quad \bar X_n^t=\bar S_t X_n=\Pi(X_n^t)=X_n^t\circ(g_n^t)^{-1}.
  \end{equation*}
  By Lemma \ref{lem:Fpres}, we have
  $g_n^t\in\Gr_\kappa$ for some $\kappa$ which
  depends only on $M$ and $T$. The sequence
  $\{\bar X_n^t\}$ has endpoints given by $\bar
  S_t(X_\alpha)$ and $\bar S_t(X_\beta)$. Since
  $X_n\in\F^M$ and the set $\F^M$ is preserved by
  the flow of the equation and relabeling, we
  have $\bar X_n^t\in\F_0\cap\F^M=\F_0^M$ so that
  the sequence $\{\bar X_n^t\}$ is in $\F_0^M$, as
  required in the definition of $d^M$. For
  $f_{n-1}^t=g_{n-1}^t\circ f_{n-1}\circ
  (g_n^t)^{-1}$, we have
  \begin{align}
    \notag
    \norm{\bar X_{n-1}^t\circ f_{n-1}^t-\bar
      X_{n}^t}&=\norm{X_{n-1}^t\circ
      (g_{n-1}^t)^{-1}\circ
      f_{n-1}^t-X_{n}^t\circ(g_{n}^t)^{-1}}\\
    \notag
    &\leq C_M\norm{X_{n-1}^t\circ
      (g_{n-1}^t)^{-1}\circ
      f_{n-1}^t\circ(g_{n}^t)-X_{n}^t}\text{ (by \eqref{eq:xrlabnoem})}\\
    \notag
    &=C_M\norm{X_{n-1}^t\circ f_{n-1}-X_{n}^t}\\
    \notag
    &=C_M\norm{S_t(X_{n-1})\circ f_{n-1}-S_t(X_n)}\\
    \label{eq:esbefSt}
    &=C_M\norm{S_t(X_{n-1}\circ
      f_{n-1})-S_t(X_n)}\text{ (by the
      equivariance of $S_t$)}.
  \end{align}
  To use the stability result \eqref{eq:stabSt},
  we have to bound $\norm{X_{n-1}\circ f_{n-1}}$
  and $\norm{X_n}$. By Lemma \ref{lem:bdEf0M}, there
  exists $\bar M$ such that $\norm{X}\leq \bar M$
  for any $X\in\F_0^M$. Hence, $\norm{X_n}\leq\bar
  M$ as $X_n\in\F_0^M$. Since $f_{n-1}$ is a priori
  arbitrary, it may seem difficult to bound
  $\norm{X_n\circ f_{n-1}}$, and it is important to
  note here that the relabeling invariant
  pseudosemimetric $\tilde J$, see
  \eqref{eq:defJt}, would not provide us with a
  bound on this term and the following estimates in
  fact motivate the Definition
  \ref{eq:defJ}. Indeed, by \eqref{eq:defJ}, we
  obtain \eqref{eq:sumXnm1} which yields
  \begin{align*}
    \norm{X_{n-1}\circ f_n-X_n}&\leq d^M(X_\alpha,X_\beta)+2\\
    &\leq
    2\norm{X_\alpha-X_\beta}+2\quad\text{(by
      \eqref{eq:dequiv})}\\
    &\leq 4\bar M+2,
  \end{align*}
  as $X_\alpha,X_\beta\in\F_0^M$.  Therefore, by
  the triangle inequality, $\norm{X_{n-1}\circ
    f_n}\leq 5\bar M+2$ so that
  $\norm{X_{n-1}\circ f_n}$ and $\norm{X_{n}}$ are
  bounded by a constant depending only on
  $M$. Thus, we can use \eqref{eq:stabSt} and get
  from \eqref{eq:esbefSt} that
  \begin{equation*}
    \norm{\bar X_{n-1}^t\circ f_{n-1}^t-\bar
      X_{n}^t}\leq C_M\norm{X_{n-1}\circ f_{n-1}-X_n},
  \end{equation*}
where from now on $C_M$ denotes some constant dependent on $M$ and $T$.
  Similarly for $\tilde f_n^t=g_n^t\circ \tilde
  f_n\circ (g_{n-1}^t)^{-1}$, we get that
  \begin{equation*}
    \norm{\bar X_{n-1}^t-\bar
      X_{n}^t\circ \tilde f_{n}^t}\leq C_M\norm{X_{n-1}-X_n\circ \tilde f_{n}}.
  \end{equation*}
  Finally, we have
  \begin{align*}
    d^M(\bar S_tX_\alpha,\bar S_tX_\beta)&\leq
    \sum_{n=1}^N(\norm{\bar X_{n-1}^t\circ
      f_{n-1}^t-X_{n}}+\norms{\bar X_{n-1}^t-\bar
      X_n^t\circ \tilde{f}_{n}^t})\\
    &\leq C_M \sum_{n=1}^N(\norm{ X_{n-1}\circ
      f_{n-1}-X_{n}}+\norms{ X_{n-1}-
      X_n\circ \tilde{f}_{n}})\\
    &\leq C_M(d^M(X_\alpha,X_\beta)+2\epsi).
  \end{align*}
  The result follows by letting $\epsi$ tend to
  zero.
\end{proof}

\section{From Lagrangian to Eulerian coordinates}

We now introduce a second set of coordinates, the
so--called Eulerian coordinates. Therefore let us
first consider $X=(y,U,H)\in\F$. We can define
Eulerian coordinates as in \cite{HR} and also
obtain the same mappings between Eulerian and
Lagrangian coordinates (see also Figure~\ref{fig:explfig}). For completeness we will
state the results here.

\begin{figure}[h]
  \centering
  \includegraphics[width=10cm]{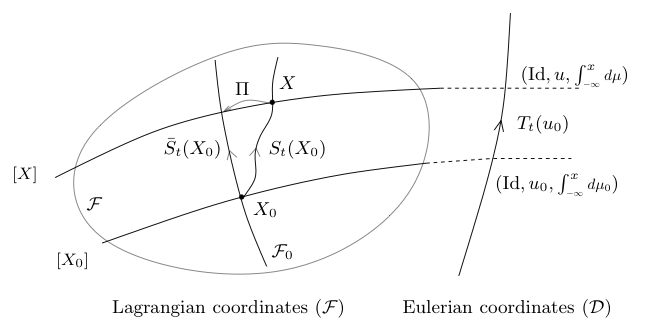}
  \caption{A schematic illustration of the
    construction of the semigroup. The set $\F$
    where the Lagrangian variables are defined is
    represented by the interior of the closed
    domain on the left. The equivalence classes
    $[X]$ and $[X_0]$ (with respect to the action
    of the relabeling group $G$) of $X$ and $X_0$,
    respectively, are represented by the
    horizontal curves. To each equivalence class
    there corresponds a unique element in $\F_0$
    and $\D$ (the set of Eulerian variables). The
    sets $\F_0$ and $\D$ are represented by the
    vertical curves.}
    \label{fig:explfig}
\end{figure}

\begin{definition}
  \label{def:D}
 The set $\D$ consists of all pairs $(u,\mu)$ such that 
\begin{enumerate}
 \item 
  $u\in H^1(\Real)$, and 
\item
$\mu$ is a positive Radon measure whose absolutely continuous part $\mu_{ac}$  satisfies 
\begin{equation}
 \mu_{ac}=u^2+u_x^2.
\end{equation}
\end{enumerate}
\end{definition}

We can define a mapping, denoted by $L$, from $\D$
to $\F_0$:
\begin{definition}
 For any $(u,\mu)$ in $\D$ let,
\begin{equation}
\left\{ 
\begin{aligned}
  y(\xi)&=\sup\{y \mid   \mu((-\infty,y))+y<\xi\},\\
  H(\xi)& =\xi-y(\xi),\\ 
  U(\xi)&=u\circ y(\xi).
 \end{aligned}
\right.
\end{equation}
Then $(y,U,H)\in\F_0$, and we denote by $L:\D\to \F_0$ the map which to any $(u,\mu)$ associates $X\in\F_0$. 
\end{definition}
Thus from any initial data $(u_0,\mu_0)\in\D$, we can construct a solution of \eqref{equivsys} in $\F$ with initial data $X_0=L(u_0,\mu_0)\in\F_0$. It remains to go back to the original variables, which is the purpose of the mapping $M$ defined as follows: 

\begin{definition}
 Given any element $X $ in $\F_0$, then $(u,\mu)$ defined as follows 
\begin{equation}
u(x)=U(\xi) \text{ for any } \xi \text{ such that } x=y(\xi), 
\end{equation}
\begin{equation}
 \mu=y_\#(\nu d\xi), 
\end{equation}
belongs to $\D$. We denote by $M:\F_0\to \D$ the map which to any $X$ in $\F_0$ associates $(u,\mu)$.
\end{definition}
In fact, $M$ can be seen as a map from $\F/\Gr\to D$, as any two elements belonging to the same equivalence class in $\F$ are mapped to the same element in $\D$ (cf.~\cite{HR}).  
Moreover, identifying elements belonging to the same equivalence class, the mappings $L$ and $M$ are invertible and 
\begin{equation}\label{eq:inversma}
 L\circ M=\id_{\F/\Gr}, \quad \text{ and }\quad M\circ L=\id_{\D}.
\end{equation}

We will now use these mappings for defining also a Lipschitz metric on $\D$. 

\begin{definition}
  Let
  \begin{equation}
    T_t:=MS_tL \colon \D\rightarrow \D.
  \end{equation}
\end{definition}

Next we show that $T_t$ is a Lipschitz continuous
semigroup by introducing a metric on
$\D$. Using the map $L$ we can transport the topology from $\F_0$ to $\D$. 

\begin{definition}
  Define the metric $d_\D\colon
  \D\times \D \rightarrow
  [0,\infty)$ by
  \begin{equation}
    d_\D((u,\mu),(\tilde{u},\tilde{\mu}))=d(L(u,\mu),L(\tilde{u},\tilde{\mu})).
  \end{equation}
\end{definition}

The Lipschitz stability of the semigroup $T_t$ follows then naturally from Theorem~\ref{th:stab}. It holds on sets of bounded energy, which are given as follows.

\begin{definition}
\label{def:DM}
 Given $M>0$, we define the subsets $\D^M$ of $\D$, which correspond to sets of bounded energy, as
\begin{equation}
 \D^M=\{ (u,\mu)\in\D \ | \ \mu(\Real)\leq M\}.  
\end{equation}
On the set $\D^M$ we define the metric $d_{\D^M}$ as 
\begin{equation}\label{eq:defDM}
 d_{\D^M}((u,\mu),(\tilde u, \tilde \mu))=d^M (L(u,\mu), L(\tilde u,\tilde \mu)),
\end{equation}
where the metric $d^M$ is defined as in Definition~\ref{def:metric}.
\end{definition}

Definition \ref{def:DM} is well-posed as we can check from the definition of $L$:  If $(u,\mu)\in \D^M$, then $L(u,\mu)\in \F_0^M$. 

\begin{theorem}
 The semigroup $(T_t, d_\D)$ is a continuous semigroup on $\D$ with respect to the metric $d_D$. The semigroup is Lipschitz continuous on sets of bounded energy, that is: Given $M>0$ and a time interval $[0,T]$, there exists a constant $C_M$, which only depends on $M$ and $T$ such that for any $(u,\mu)$ and $(\tilde u,\tilde \mu)$ in $\D^M$, we have 
\begin{equation}
 d_{D^M}(T_t(u,\mu),T_t(\tilde u, \tilde\mu))\leq C_Md_{\D^M}((u,\mu),(\tilde u, \tilde \mu))
\end{equation}
 for all $t\in [0,T]$.
\end{theorem}

\begin{proof}
First we prove that $T_t$ is a semigroup. Since $\bar{S}_t$ is a mapping from $\F_0$ to $\F_0$, we have 
 \begin{equation*}
    T_{t}T_{t'}=M\bar S_tLM\bar S_{t'}L=M\bar S_t\bar S_{t'}L=M\bar S_{t+t'}L=T_{t+t'}
  \end{equation*}
  where we also used \eqref{eq:inversma} and the
  semigroup property of $\bar S_t$. We now prove
  the Lipschitz continuity of $T_t$. By using
  Theorem \ref{th:stab}, we obtain that 
  \begin{align}\nn
    d_{\D^M}(T_t(u,\mu),T_t(\tilde{u},\tilde{\mu}))
    & = d^M( LM\bar S_tL(u,\mu)
    LM\bar S_tL(\tilde{u},\tilde{\mu}))\\\nn
    & = d^M( \bar S_t L(u,\mu), \bar S_t L(\tilde{u},\tilde{\mu}))\\
    & \leq C_M d^M( L(u,\mu),
     L(\tilde{u},\tilde{\mu}))\\\nn & = C_M
    d_{\D^M}((u,\mu),(\tilde{u},\tilde{\mu})).
  \end{align}
\end{proof}

By a weak solution of the Camassa--Holm equation
we mean the following.
\begin{definition}
  Let $u\colon\Real\times\Real \rightarrow \Real$ that
  satisfies
  \begin{enumerate}
  \item $u\in L^\infty ([0,\infty), H^1(\Real))$,
  \item the equations
    \begin{equation}\label{weak1}
      \iint_{\Real_+\times \Real}-u(t,x)\phi_t(t,x)+(u(t,x)u_x(t,x)+P_x(t,x))\phi(t,x)dxdt = \int_\Real u(0,x)\phi(0,x)dx,
    \end{equation}
    and
    \begin{equation}\label{weak2}
      \iint_{\Real_+\times \Real}(P(t,x)-u^2(t,x)-\frac{1}{2} u_x^2(t,x))\phi(t,x)+P_x(t,x)\phi_x(t,x)dxdt=0,
    \end{equation}
  \end{enumerate}
  hold for all $\phi\in C_0^\infty
  ([0,\infty),\Real)$. Then we say that $u$ is a weak
  global solution of the Camassa--Holm equation.
\end{definition}

\begin{theorem}
  Given any initial condition
  $(u_0,\mu_0)\in\D$, we denote
  $(u,\mu)(t)=T_t(u_0,\mu_0)$. Then $u(t,x)$ is a
  weak, global solution of the Camassa--Holm
  equation.
\end{theorem}

\begin{proof}
  After making the change of variables
  $x=y(t,\xi)$ we get on the one hand
  \begin{align}\nn
    -\iint_{\Real_+\times\Real} & u(t,x)\phi_t(t,x)dxdt
    = - \iint_{\Real_+\times
      \Real}u(t,y(t,\xi))\phi_t(t,y(t,\xi))y_\xi(t,\xi)d\xi
    dt\\\nn & =- \iint_{\Real_+\times \Real}
    U(t,\xi)[(\phi(t,y(t,\xi))_t-\phi_x(t,y(t,\xi)))y_t(y,\xi)]y_\xi(t,\xi)d\xi
    dt\\\nn
    & = -\iint_{\Real_+\times \Real}[U(t,\xi)y_\xi(t,\xi)(\phi(t,y(t,\xi)))_t-\phi_\xi(t,y(t,\xi))U(t,\xi)^2]d\xi dt\\
    & = \int_\Real
    U(0,\xi)\phi(0,y(0,\xi))y_\xi(0,\xi)d\xi\\\nn
    & \quad + \iint_{\Real_+\times \Real}
    [U_t(t,\xi)y_\xi(t,\xi)+U(t,\xi) y_{\xi
      t}(t,\xi)]\phi(t,y(t,\xi))d\xi dt \\\nn &
    \quad +\iint_{\Real_+\times \Real}
    U^2(t,\xi)\phi_\xi(t,y(t,\xi))d\xi dt\\\nn & =
    \int_\Real u(0,x)\phi(0,x)dx\\\nn & \quad
    -\iint_{\Real_+\times \Real}
    (Q(t,\xi)y_\xi(t,\xi)+U_\xi(t,\xi)U(t,\xi))\phi(t,y(t,\xi))d\xi
    dt,
  \end{align}
  while on the other hand
  \begin{align}\nn
    \iint_{\Real_+\times \Real} & (u(t,x)u_x(t,x)+P_x(t,x))\phi(t,x)dxdt\\
    & = \iint_{\Real_+\times
      \Real}(U(t,\xi)U_\xi(t,\xi)+P_x(t,y(t,\xi))y_\xi(t,\xi))\phi(t,y(t,\xi))d\xi
    dt\\\nn & = \iint_{\Real_+\times \Real}
    (U(t,\xi)U_\xi(t,\xi)+Q(t,\xi)y_\xi(t,\xi))\phi(t,y(t,\xi))d\xi dt,
  \end{align}
  which shows that \eqref{weak1} is fulfilled.
 Equation \eqref{weak2} can be shown analogously
  \begin{align}\nn
    \iint_{\Real_+\times \Real}&
    P_x(t,x)\phi_x(t,x)dxdt\\\nn
    & = \iint_{\Real_+\times \Real} Q(t,\xi)y_\xi(t,\xi)\phi_x(t, y(t,\xi))d\xi dt\\
    &= \iint_{\Real_+\times \Real}
    Q(t,\xi)\phi_\xi(t,y(t,\xi)) d\xi dt\\\nn & =
    -\iint_{\Real_+\times
      \Real}Q_\xi(t,\xi)\phi(t,y(t,\xi)) d\xi dt\\\nn
    & = \iint_{\Real_+\times \Real} [\frac{1}{2}
    H_\xi(t,\xi)+(\frac{1}{2}U^2(t,\xi)-P(t,\xi))y_\xi(t,\xi)]\phi(t,y(t,\xi))d\xi
    dt\\\nn &= \iint_{\Real_+\times \Real}[\frac{1}{2}
    u_x^2(t,x)+u^2(t,x)-P(t,x)]\phi(t,x) dx dt.
  \end{align}
  In the last step we used the following
  \begin{align}
   \int_\Real u^2+u_x^2 dx & =\int_\Real    u^2\circ y y_\xi+u^2_\xi\circ yy_\xi d\xi\\ \nn
    & =\int_{\{\xi\in\Real \mid y_\xi(t,\xi)>0\}} U^2y_\xi+\frac{U_\xi^2}{y_\xi}d\xi=\int_\Real H_\xi d\xi. 
  \end{align}
 For almost every $t\in\Real_+$ the set $\{\xi\in\Real \mid y_\xi(t,\xi)>0\}$  is of full measure and hence
\begin{equation}
 \int_\Real u^2+u_x^2 dx=\int_\Real H_\xi d\xi,
\end{equation}
which is bounded by a constant for all times. Thus we proved that $u$ is a weak solution of the Camassa--Holm equation. 
\end{proof}

\section{The topology on $\D$} \label{sec:top}

\begin{proposition}
  The mapping
  \begin{equation}
    u\mapsto  (u,(u^2+u_x^2) dx)
  \end{equation} 
  is continuous from $H^1(\Real)$ into
  $\D$. In other words, given a sequence
  $u_n\in H^1(\Real)$ converging to $u\in H^1(\Real)$, 
  then $(u_n,(u_n^2+u_{nx}^2)dx)$ converges to $(u,(u^2+u_x^2)
  dx)$ in $\D$.
\end{proposition}

\begin{proof}
  Let $X_n=(y_n, U_n, H_n)= L(u_n,
  (u_n^2+u_{nx}^2)dx)$ and
  $X=(y,U,H)=L(u,(u^2+u_x^2)dx)$. Then as in 
  the proof of \cite[Proposition 5.1]{HR} one can
  show that
  \begin{equation}
    X_n\to X \text{ in } E.
  \end{equation}
  Hence using \eqref{eq:dequiv}, we get that
  $\lim_{n\to\infty} d(X_n,X)=0$.
\end{proof}

\begin{proposition}
  Let $(u_n, \mu_n)$ be a sequence in
  $\D$ that converges to $(u,\mu)$ in
  $\D$. Then
  \begin{equation}
    u_n\rightarrow u \text{ in } L^\infty(\Real) \text{ and } \mu_n \overset{\ast}{\rightharpoonup}\mu.
  \end{equation}
\end{proposition}

\begin{proof}

  Let $X_n=(y_n, U_n, H_n)=L(u_n,\mu_n)$
  and $X=(y,U,H)=L(u,\mu)$ .  By the
  definition of the metric $d_\D$, we
  have $\lim_{n\to\infty}d(X_n,X)=0$. Using \eqref{eq:dequiv}, we immediately obtain 
  that
  \begin{equation}
    X_n \to X \text{ in } L^\infty(\Real).
  \end{equation}
  The rest can be proved as in \cite[Proposition
  5.2]{HR}.
\end{proof}

\noindent{\bf Acknowledgments.} 
K. G. gratefully acknowledges the hospitality of
the Department of Mathematical Sciences at the
NTNU, Norway, creating a great working environment
for research during the fall of 2009.

\end{document}